\documentclass[12pt]{amsart}

\usepackage{amssymb}
\usepackage{amsfonts}
\usepackage{amsmath}
\usepackage{amscd}
\usepackage[all,cmtip]{xy}
\usepackage{amsthm}
\usepackage{setspace}
\usepackage{enumerate}

\usepackage{amsfonts}
\usepackage{amsmath}
\usepackage{amscd}
\usepackage{amsthm}
\usepackage{amssymb}
\usepackage{delarray}
\usepackage[all,cmtip]{xy}
\usepackage[dvips]{graphics}
\usepackage{epsfig}
\usepackage{mathrsfs}
\usepackage{color}
\usepackage{setspace}
\usepackage{enumerate}
\usepackage{color}
\usepackage{url}

\DeclareMathSymbol{\minus} {\mathord}{operators}{"2D}

\numberwithin{equation}{section}
\newcommand\GLn{\textrm{GL}_n}
\newcommand\SLn{\textrm{SL}_n}
\newcommand\SLt{\textrm{SL}_2}
\newcommand\ZZ{\mathbb Z}
\newcommand\sgn{\textrm{sgn}}

\renewcommand\bar\overline

\newtheorem{thm}{Theorem}[section]
\newtheorem{lemma}[thm]{Lemma}
\newtheorem{prop}[thm]{Proposition}

\newtheorem{fact}[thm]{Fact}

\theoremstyle{definition}
\newtheorem{defn}[thm]{Definition}

\theoremstyle{remark}

\theoremstyle{claim}

\begin{document}

\title{Generic thinness in finitely generated subgroups of $\SLn(\mathbb Z)$}
\author{Elena Fuchs}
\address{Elena Fuchs, Department of Mathematics, University of Illinois, 1409 W. Green St, Urbana IL 61801}
\email{lenfuchs@illinois.edu}
\author{Igor Rivin}
\address{Igor Rivin, Temple University, Mathematics Department, Wachman Hall, 1805 N. Broad St, Philadelphia PA 19122} 
\address{Igor Rivin, School of Mathematics and Statistics, University of St. Andrews, St Andrews KY16 9SS, Scotland}
\email{igor.rivin@st-andrews.ac.uk}

\maketitle

\begin{abstract}
We show that for any $n\geq 2$, two elements selected uniformly at random from a \emph{symmetrized} Euclidean ball of radius $X$ in $\SLn(\mathbb Z)$ will generate a thin free group with probability tending to $1$ as $X\rightarrow \infty.$  This is done by showing that the two elements will form a ping-pong pair, when acting on a suitable space, with probability tending to $1$.  On the other hand, we give an upper bound less than $1$ for the probability that two such elements will form a ping-pong pair in the usual Euclidean ball model in the case where $n>2$.
\end{abstract}

\section{Introduction}
Given a subgroup $\Gamma$ of $\GLn(\ZZ)$ with Zariski closure $\overline \Gamma$ in $\GLn(\mathbb C)$, $\Gamma$ is called a {\it thin group} if it is of infinite index in $\overline \Gamma\cap \GLn(\ZZ)$.

In this paper, we investigate the question of whether a generic finitely generated subgroup of $\SLn(\ZZ)$ is thin.  Our notion of genericity is described via a Euclidean model, as we discuss below.  This question is motivated in part by recent developments in number theory (\cite{bgs}, \cite{salsarnak}, \cite{SalehiVarju}, etc.) which have made approachable previously unsolved arithmetic problems involving thin groups (see, for example, \cite{Fuchsbull} and \cite{Kontbull} for an overview).  Given these new ways to handle such groups in arithmetic settings, it has become of great interest to develop a better understanding of thin groups in their own right: for example, \cite{FMS}, \cite{SV}, and \cite{BT} have answered the question of telling whether a given finitely generated group is thin (given in terms of its generators) in various settings.  Our question of asking whether such a group is generically thin of a similar flavor, but one has more freedom in avoiding the more difficult cases (as we show, the generic $2$-generator subgroup of $\textrm{SL}_n(\mathbb Z)$ is free when generic is defined appropriately, and yet in most concrete examples investigated thus far, one does not have freeness).

Note that similar questions regarding in particular subgroups of $\SLn(\mathbb Z)$ generated by random elements of {\it combinatorial} height at most $X$ -- i.e. elements obtained by a random walk of length $X$ on the Cayley graph $\textrm{Cay}(\SLn(\mathbb Z), S)$ where $S$ is a fixed finite set of generators of $\SLn(\mathbb Z)$ -- have been previously addressed by Rivin in \cite{R1,R3} and by Aoun in \cite{Aoun}.  In fact, Rivin considers a much broader family of lattices in semisimple Lie groups beyond $\SLn(\mathbb Z)$, and Aoun's results apply also to finitely generated non-virtually solvable subgroups of $\textrm{GL}(V)$ where $V$ is a finite dimensional vector space over an arbitrary local field $K$.

Aoun's result in the context of $\textrm{SL}_n(\mathbb Z)$ that any two independent random walks on $\SLn(\mathbb Z)$ generate a free group (we give a short proof of this here as Theorem \ref{aounthm}), implies that with this combinatorial definition of genericity, a finitely generated subgroup of $\SLn(\ZZ)$ is generically of infinite index in $\SLn(\ZZ)$ if $n\geq 3$ (Aoun shows this by proving that any two such random walks will yield a ping-pong pair; see Section~\ref{pingpongsection} for a definition of ping-pong, and Theorem \ref{aounthm} for a proof of a version of Aoun's result).  Combining this with Rivin's result in \cite{R1} that in the combinatorial model a generic finitely generated subgroup of $\SLn(\ZZ)$ is Zariski dense in $\SLn(\mathbb C)$, we have that thinness is generic in the combinatorial setup.  

It is hence perhaps reasonable to expect that thinness should be generic in the following Euclidean model.  Let $G=\textrm{SL}_n(\mathbb Z)$, and let $B_X$ denote the set of all elements in $G$ of norm at most $X$, where norm is defined as
\begin{equation}\label{height}
||\gamma||^2:= \lambda_{\textrm{max}}(\gamma^{t}\gamma)
\end{equation}
where $\lambda_{\max}$ denotes the largest eigenvalue.  Our task in this paper is to choose two elements $g_1,g_2$ uniformly at random from $B_X$ and to consider
$$\lim_{X\rightarrow\infty}\mu_X(\{g=(g_1,g_2)\in G^2\; |\; \Gamma(g) \mbox{ is of infinite index in } G\})$$
where $\Gamma(g)=\langle g_1,g_1^{-1},g_2,g_2^{-1}\rangle$, and $\mu_X$ is the measure on $G\times G$ induced by the normalized counting measure on $B_X^2$.  If the above limit is $1$, we say that the generic subgroup of $G$ generated by two elements is infinite index in $G$.  In general, we say that the generic subgroup of $G$ generated by two elements has some property $P$ if
\begin{equation*}
\lim_{X\rightarrow\infty}\mu_X(\{g=(g_1,g_2)\in G^2\; |\; \Gamma(g) \mbox{ has property $P$}\})=1.
\end{equation*}
It is a result of the second author (\cite{R2}) that in this model, two randomly chosen elements \emph{do} generate a Zariski-dense subgroup, and 
might expect that, just as in the combinatorial setting, two randomly chosen elements of $G$ in the Euclidean model will also form a ping-pong pair with probability tending to $1$ (i.e. that the generic $2$-generator subgroup of $G$ is generated by a ping-pong pair in particular).   Surprisingly, we use Breuillard-Gelander's \cite{BG} characterization of ping-pong for $\SLn(\mathbb R)$ over projective space to show that while this is the case for $n=2$, it is not true if $n>2$ (see Theorem~\ref{toptwoth}), and so further work must be done to prove that thinness is generic in this model, if it is in fact the case.

However, if one ``symmetrizes" the ball of radius $X$ in a natural way, by imposing a norm bound on both the matrix and its inverse, we show that two elements chosen at random in such a modified model \emph{will,} in fact, be a ping-pong pair over a suitable space with probability tending to $1$, and this enables us to show that, in this modified setup, the generic subgroup of $\textrm{SL}_n(\ZZ)$ generated by two elements is thin (our methods extend to any arbitrary finite number of generators in a straightforward way, as well).  This modified Euclidean model is identical to the one described above, but $B_X$ will be replaced by
\begin{equation}\label{modball}
B_X'(G):=\{g\in G\;|\; g,g^{-1}\in B_X\},
\end{equation}
and the measure $\mu_X$ is replaced by $\mu'_X$, the normalized counting measure on $(B_X')^2$.    With this notation, we show the following.

\begin{thm}\label{thinsln}
Let $G=\mathrm{SL}_n(\ZZ)$ where $n\geq 2$, and let $B_X'(G)$ and $\mu_X'$ be as above.  Then we have
\begin{equation*}
\lim_{X\rightarrow\infty}\mu'_X(\{(g_1,g_2)\in (B_X'(G))^2\; |\; \langle g_1,g_2\rangle \mbox{\emph{ is thin}}\})=1
\end{equation*}
\end{thm}

The key tool in all of the $n>2$ cases is Section~4 of \cite{BG}.  We remark that it is very natural to consider the region $B_X'$, rather than the usual ball $B_X$: it is in fact a more suitable analog of Aoun's combinatorial setup, in which an element of combinatorial height $X$ has inverse whose combinatorial height is also $X$ (unlike the Euclidean ball model, in which an element of norm $X$ can have inverse of much larger norm).

It should be remarked that choosing the symmetrized ball is morally similar to choosing the condition number of $\gamma$ as the measure of size. Recall that the condition number of a matrix $A$ is the ratio of the largest to the smallest singular value, which is, in our notation, equal to $\|A\| \|A^{-1}\|.$ 

We also address the question of ``How thin is thin?'' Namely, in the special case of $\textrm{SL}_2$ we show that as we pick our pairs of elements out of bigger and bigger balls, the Hausdorff dimension of the limit set of the subgroup they generate becomes smaller and smaller. In higher rank, the Hausdorff dimension is less tractable, but what can be shown  is that the top \emph{Lyapunov exponent} of the random pair becomes larger and larger (it is well-known that the top Lyapunov exponent and the Hausdorff dimension are, essentially, the inverses of one another).

Our paper is organized as follows.  In Section~\ref{sl2sec} we prove that the generic subgroup of $\SLt(\ZZ)$ is free and is of arbitrarily small Hausdorff dimension, which together show that the generic subgroup of $\SLt(\ZZ)$ is thin.  In Section~\ref{sln}, we consider subgroups of $\SLn$ where $n>2$.  In Section~\ref{epscontract} we show that two randomly chosen elements of $\SLn(\ZZ)$ will be a ping-pong pair with probability tending to $0<\alpha<1$ in the Euclidean model described above.  In Section~\ref{modified} we show that the generic subgroup of $\textrm{SL}_n(\mathbb Z)$ is thin in the modified Euclidean model.  This is done by showing that two randomly chosen elements of $\SLn(\ZZ)$ will be a ping-pong pair in a suitable space $\mathbb P(\bigwedge^k(\mathbb R^n))$ with probability tending to $1$ (in fact, for $n=3$ one sets $k=1$ and so the space is simply $\mathbb P(\mathbb R^n)$).  For this, we use results from \cite{BG},\cite{EM}, and \cite{GO}.  In Section~\ref{wellroundedsec} we show the well-roundedness of certain sequences of sets which is necessary in applying the results from \cite{EM} and \cite{GO} in the preceding section. In Section \ref{lyapsec} we discuss the Lyapunov exponents, and in Section \ref{future} we discuss some open questions.

{\bf Acknowledgements:} The question of whether thin is generic was first posed to the authors by Peter Sarnak, and we are very grateful to him for many insightful conversations and correspondences on the subject.  We also thank Hee Oh, Alireza Salehi-Golsefidy, Emmanuel Breuillard, Giulio Tiozzo and Curt McMullen for several helpful conversations.  Finally, this work began during our stay at the Institute for Advanced Study in Princeton, and we are grateful for the institute's hospitality.

\section{Subgroups of $\textrm{SL}_2(\ZZ)$}\label{sl2sec}

In this section, we prove the following.

\begin{thm}\label{sl2}
Let $G=\mathrm{SL}_2(\ZZ)$, and let $\Gamma(g)$ and $\mu_X$ be as above.  Then we have
\begin{equation*}
\lim_{X\rightarrow\infty}\mu_X(\{g=(g_1,g_2)\in G^2\; |\; \Gamma(g) \mbox{\emph{ is thin}}\})=1
\end{equation*}
\end{thm}

In this case, we are able to prove generic thinness even in the usual Euclidean ball model, which we do next.  We separate the $2$-dimensional case from the other cases for several reasons: one is that in this case we have the very natural action of $\textrm{SL}_2$ on the upper half plane to work with, and two is that the general strategy is the same as in the higher-dimensional cases yet more straightforward.  The idea is to show generic freeness using a ping-pong argument which follows essentially from studying the generators' singular values (i.e. eigenvalues of $g_i^tg_i$) and using equidistribution results.

Specifically, Theorem~\ref{sl2} will follow from three lemmas which we prove below. Lemma~\ref{disjointcircles} will imply Lemma~\ref{sl2schottky}, that the generic $\Gamma(g)$ is free.  In fact, it is {\it Schottky} -- namely, one expects that there exist four disjoint and mutually external circles $C_1,C_2,C_3, C_{4}$ in $\mathbb H$ such that for $1\leq i\leq 2$ the generator $g_i$ of $\Gamma$ maps the exterior of $C_i$ onto the interior of $C_{i+2}$, and the generator $g_i^{-1}$ of $\Gamma$ maps the exterior of $C_{i+2}$ onto the interior of $C_i$.  We then use Lemma~\ref{sl2schottky} to show that generically the limit set of $\Gamma$ acting on $\mathbb H$ has arbitrarily small Hausdorff dimension in Lemma~\ref{haussl2}, which immediately implies Theorem~\ref{sl2}, and in fact shows that the generic subgroup generated by two elements in $\SLt(\ZZ)$ is ``arbitrarily thin" (Hausdorff dimension is a natural measure of thinness: if it is not $1$, then it is thin, and the smaller the Hausdorff dimension, the thinner the group).

Let
\begin{equation*}
B_X:=\{\gamma\in \SLt(\ZZ)\; |\; ||\gamma||\leq X\}
\end{equation*}
and note that $|B_X|\sim c\cdot X^2$ for some constant $c$ (see \cite{DRS} and \cite{D}).  Also, for a fixed $T>0$ we have
\begin{equation*}
|\{\gamma\in B_X\; {\mbox{s.t. }}\; |\textrm{tr}(\gamma)|<T\}|\ll_{\epsilon}T\cdot X^{1+\epsilon}
\end{equation*}
so that for any $T$
\begin{equation}\label{tracelarge}
\lim_{X\rightarrow\infty}\mu_X(\{\gamma\in G\; | \; |\textrm{tr}(\gamma)|>T\})=1
\end{equation}
In particular, we have $\lim_{X\rightarrow\infty}\mu_X(\{\gamma\in G\; | \; \textrm{tr}(\gamma)>2\})=1$.  So, with probability tending to $1$, each of the generators $g_i$ of $\Gamma(g)$ will have two fixed points on the boundary $S$ of $\mathbb H$.  Write
\begin{equation}
g_i=\left(\begin{array}{ll}
           a_i & b_i\\
c_i & d_i\\
          \end{array}\right)
\end{equation}
Then the fixed points of $g_i$ are
\begin{equation*}
\frac{(d_i-a_i)\pm\sqrt{(a_i+d_i)^2-4}}{2c_i}
\end{equation*}
which, since the trace of $g_i$ is large, approaches $(d_i-a_i\pm (a_i + d_i))/(2c_i) = d_i/c_i, \; -a_i/c_i$ as $X$ goes to infinity.  One of these points is attracting -- call it $\alpha_i$ -- and the other is repelling -- call it $\beta_i$. The distance between these points is $|(a_i+d_i)/c_i|$, which is large with high probability.  Furthermore, there exists a circle $C_i$ containing $\alpha_i$ and a circle $C_{i+2}$ containing $\beta_i$, both of radius $1/|c_i|$, (the isometric circles) such that $g_i$ maps the exterior of $C_i$ onto the interior of $C_{i+2}$ and $g_i^{-1}$ maps the exterior of $C_{i+2}$ onto the interior of $C_i$.  Note that as $X$ tends to infinity, the probability that the radii of these circles are small is large, and that generically $C_i$ and $C_{i+2}$ are disjoint since $g_i$ is hyperbolic with high probability.

So far we have selected $g_1$ uniformly at random out of a ball of radius $X$.  As discussed above, as $X\rightarrow\infty$, the probability that $g_1$ is hyperbolic of large trace tends to $1$.  Now we select a second element $g_2$ uniformly at random out of a ball of radius $X$, also of large trace with probability tending to $1$.  Let $C_1,C_3,C_2,C_4$ be defined as above.  We would like to show that as $X\rightarrow\infty$, the probability that the circles $C_1,C_2,C_3,C_4$ are mutually external and disjoint (i.e. that $g_1$ and $g_2$ form a Schottky pair) tends to $1$.  Let $r(C_i)$ denote the radius of $C_i$.  Note that for any $\epsilon>0$, we have
$$\lim_{X\rightarrow\infty}\mu_X(\{(g_1,g_2)\in G^2\; | \; \max_i(r(C_i))<\epsilon\})=1$$
so our desired statement about the disjointness of the circles $C_i$ will follow from the following lemma.
\begin{lemma}\label{disjointcircles} For any pair $(g_1,g_2)\in G^2$ such that $|\textrm{tr}(g_i)|>2$ for $i=1,2$, let $\alpha_i$ and $\beta_i$ denote the fixed points of $g_i$. Let $$d_{\min}(g_1,g_2):= \min(d(\alpha_1,\beta_2),d(\alpha_1,\alpha_2),d(\beta_1,\beta_2),d(\beta_1,\alpha_2)),$$
where $d(\cdot,\cdot)$ denotes hyperbolic distance.  Then for any $r>0$ we have 
\begin{equation}\label{wantshow}
\lim_{X\rightarrow\infty}\mu_X(\{(g_1,g_2)\in G^2\; | \; |\textrm{tr}(g_i)|>2, \,d_{\min}(g_1,g_2)>r\})=1.
\end{equation}
\end{lemma}
\begin{proof}
From (\ref{tracelarge}) we have that the probability that the traces of $g_i$ are different and greater than $2$ tends to $1$.  We may therefore restrict to nonconjugate hyperbolic pairs $(g_1,g_2)$ in proving Lemma~\ref{disjointcircles}.  Specifically, it suffices to show that
\begin{equation}\label{wantshow2}
\lim_{X\rightarrow\infty}\mu_X(\{(g_1,g_2)\in G^2\; | \; \textrm{tr}(g_1)\not=\textrm{tr}(g_2),\, d_{\min}(g_1,g_2)>r\})=1
\end{equation}
Note now that to every pair of nonconjugate elements $g_1,g_2\in G$ one can associate a unique pair of distinct closed geodesics $L_1,L_2$ on $G\backslash\mathbb H$ fixed by $g_1$ and $g_2$ respectively.  Furthermore, the length $\ell(L_i)$ is
$$\ell(L_i)=\left((|\textrm{tr}(g_i)|+\sqrt{\textrm{tr}^2(g_i)-4})/2\right)^2$$
for $i=1,2$.  Therefore our measure $\mu_X$ on $G$ also induces a measure on the set $S'$ of closed geodesics on $G\backslash\mathbb H$, and we have that for a fixed $T>0$
$$\lim_{X\rightarrow\infty}\mu_X(\{L\in S' \;|\; \ell(L)>T)=1$$
from (\ref{tracelarge}).  With this in mind, (\ref{wantshow2}) follows from the equidistribution of long closed geodesics\footnote[1]{The equidistribution of long closed geodesics was first proven spectrally by Duke-Rudnick-Sarnak in \cite{DRS} and ergodic-theoretically by Eskin-McMullen in \cite{EM} shortly after.} which we summarize below from \cite{sardav}.

To each closed geodesic $L$ on $G\backslash\mathbb H$ one associates the measure $\nu_{L}$ on $G\backslash\mathbb H$ which is arc length supported on $L$.  Let $\nu=\frac{3}{\pi}\cdot\frac{dx\,dy}{y^2}$.  For a finite set $S$ of closed geodesics, let $\ell(S)=\sum_{L\in S} \ell(L)$ where $\ell(L)$ is the length of $L$, and define the measure $\nu_S$ on $G\backslash\mathbb H$ by
$$\nu_S=\frac{1}{\ell(S)}\cdot\sum_{L\in S}\nu_L$$
Let $S'$ be the set of all closed geodesics on $G\backslash\mathbb H$, and let $S'(t)=\{L\in S'\; | \; \ell(L)<t\}$.  Then we have (see \cite{DRS} or \cite{EM}) that as $t\rightarrow\infty$ the measures $\nu_{S'(t)}\rightarrow\nu$.  Therefore, given the association of pairs $(g_1,g_2)$ in (\ref{wantshow2}) with pairs of closed geodesics, we have that the pairs of fixed points $(\alpha_i,\beta_i)$ of $g_i$ chosen uniformly at random from a ball of radius $X$ also equidistribute as $X\rightarrow\infty$ as desired.
\end{proof}

Lemma~\ref{disjointcircles} implies that generically the isometric circles associated to $g_1$ and $g_2$ are disjoint.  This is precisely what is needed for $\Gamma$ to be Schottky, and so the following lemma is immediate.

\begin{lemma}\label{sl2schottky}
Let $G=\mathrm{SL}_2(\ZZ)$, and let $\Gamma(g)$ and $\mu_X$ be as before.  Then we have
\begin{equation*}
\lim_{X\rightarrow\infty} \mu_X(\{g\in G^k \mbox{ s.t. } \Gamma(g) \mbox{\emph{ is Schottky}}\})=1
\end{equation*}
\end{lemma}

In other words, as $X\rightarrow\infty$, the picture of the fixed points and isometric circles of $g_1$ and $g_2$ is generically as in figure~\ref{pingpong}: namely, the isometric circles are disjoint.  Since the generic group $\Gamma$ generated by $g_1$ and $g_2$ is Schottky, it is in fact free.  

\begin{figure}[htp]
\centering
\includegraphics[height = 50 mm]{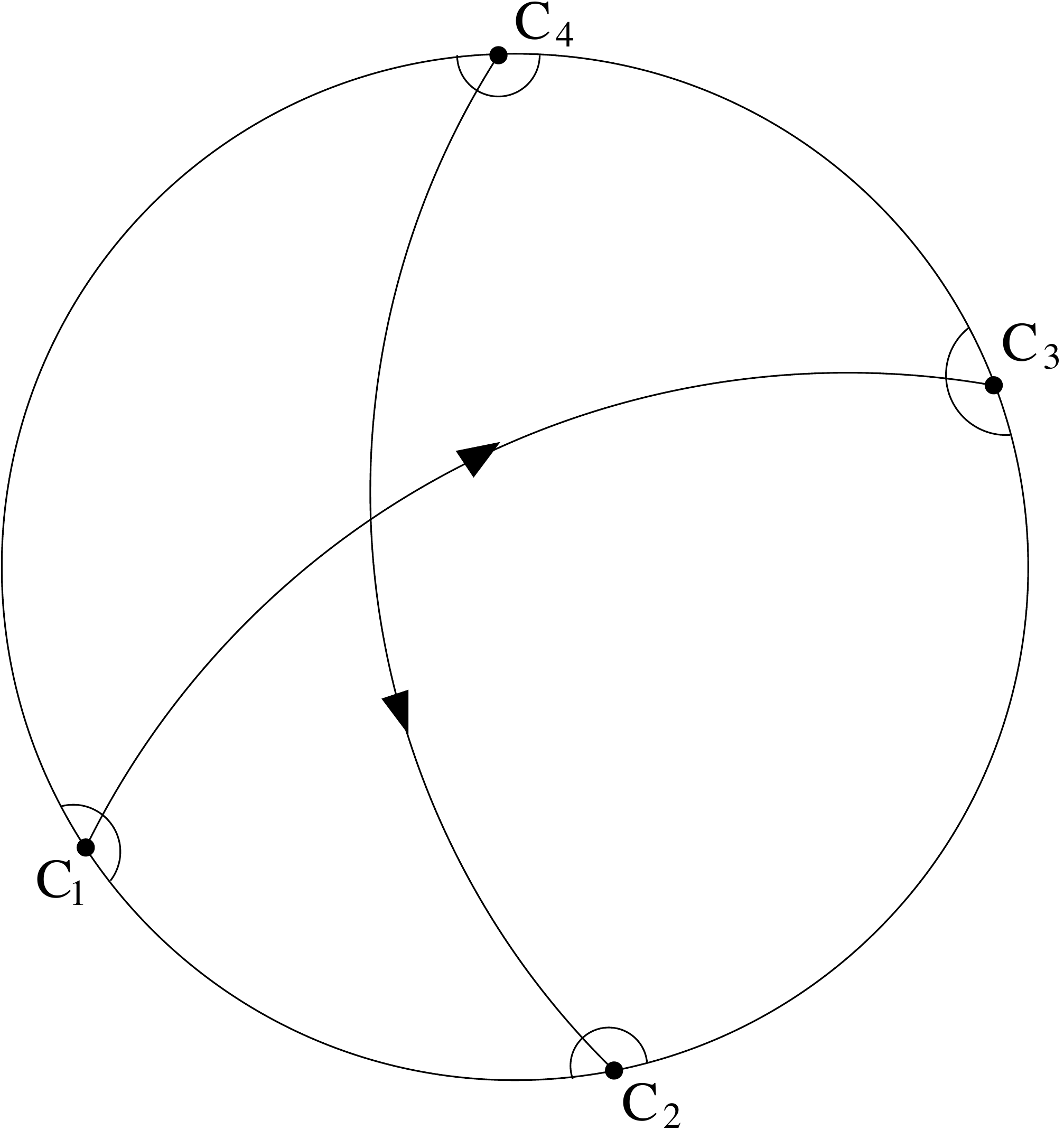}
\caption{The generic picture of a 2-generator subgroup of $\SLt(\ZZ)$ acting on $\mathbb H$.}\label{pingpong}
\end{figure}

To show that $\Gamma$ will be infinite index in $G$ with high probability, we will use the fact that it is Schottky with high probability to obtain upper bounds on the Hausdorff dimension of its limit set, i.e. the critical exponent of 
$$\sum_{\gamma\in\Gamma}e^{-\delta\cdot d(x,\gamma y)},$$
showing that it is arbitrarily small with high probability.
\begin{lemma}\label{haussl2}
Let $g_i$, $\Gamma=\Gamma(g)$ and $\mu_X$ be as above, and let $\delta(\Gamma)$ denote the Hausdorff dimension of the limit set of $\Gamma$.  Then for every $\epsilon>0$ we have
\begin{equation*}
\lim_{X\rightarrow\infty}\mu_X(\{g=(g_1,g_2) \mbox{ s.t. } 0<\delta(\Gamma)<\epsilon\})=1
\end{equation*}
\end{lemma}
\begin{proof}
Let $C_1,C_2,C_3,C_{4}$ be as above, and let $I_i$ denote the $i$th coordinate of
\begin{equation*}
I=(g_1,g_2,g_1^{-1},g_2^{-1})
\end{equation*}
By Lemma~\ref{sl2schottky}, the circles $C_i$ are mutually external and disjoint, and the generators $g_i$ of $\Gamma$ map the exterior of $C_i$ onto the interior of $C_{i+2}$ while their inverses do the opposite.  We recall the following setup from \cite{B}.  Let $K=3$, let $a(j)$ denote the center of $C_j$, and let $r(j)$ denote the radius of $C_j$.  Define $\Sigma(m)$ to be the set of all sequences $(i_1,\dots,i_m)$ where
\begin{equation*}
i_1,\dots,i_m\in \{1,2,3, 4\} \mbox{ and } i_1\not=i_{2}\pm 2, \dots, i_{m-1}\not= i_m\pm 2
\end{equation*}
Define by
\begin{equation*}
I(i_1,\dots,i_m)= I_{i_1}\cdots I_{i_{m-1}}(C_{i_m})
\end{equation*}
where $(i_1,\dots,i_m)\in\Sigma(m)$, $m\geq 2$ and $I_sI_t(x)=I_s(I_t(x))$.  It is shown in \cite{B} that the limit set $\Lambda(\Gamma)$ of $\Gamma$ can then be written as
\begin{equation}\label{limitset}
\Lambda(\Gamma)=\bigcap_{m=1}^{\infty}\bigcup_{(i_1,\dots,i_m)\in\Sigma(m)}I(i_1,\dots,i_m)
\end{equation}
and its Hausdorff dimension, $\delta(\Gamma)$ has the upper bound
\begin{equation}\label{hausupper}
\delta(\Gamma)\leq \frac{-\log K}{2\cdot\log \lambda}
\end{equation}
where
\begin{equation}\label{hauslambda}
\lambda=\max_{1\leq i\not= j\leq K+1}\frac{r(i)}{|a(i)-a(j)|-r(j)}
\end{equation}
Since in our setup all radii $r(i)$ tend to zero and the distances between the centers $a(i)$ are large with high probability as $X$ tends to infinity, (\ref{hausupper}) and (\ref{hauslambda}) imply that for any $\epsilon>0$
\begin{equation*}
\lim_{X\rightarrow\infty}\mu_X(\{g=(g_1,g_2) \mbox{ s.t. } \delta(\Gamma)<\epsilon\})=1
\end{equation*}
as desired.

To see that $\delta(\Gamma)>0$ with probability tending to $1$, note that by theorems of Rivin in \cite{R2} and Kantor-Lubotzky in \cite{KL} we have that finitely generated subgroups of $\SLn(\mathbb Z)$ are Zariski dense in $\SLn(\mathbb C)$ with probability tending to $1$ (in particular with our Euclidean measure).  Therefore $\delta(\Gamma)>0$ with probability tending to $1$.
\end{proof}

\section{Subgroups of $\SLn(\ZZ)$ for $n>2$}\label{sln}

In the previous section, we considered the action of $\textrm{SL}_2$ on the upper half plane $\mathbb H$, and showed that with probability tending to $1$ the group $\Gamma(g)$ is a Schottky group via a ping-pong method on the boundary of $\mathbb H$.  In this section, we use results of Breuillard-Gelander in \cite{BG} to prove an analogous statement for $\textrm{SL}_n(\ZZ)$ for $n>2$ after changing the notion of a ball of radius $T$ somewhat.  We also discuss what happens if we consider the usual Euclidean ball model in $\textrm{SL}_n$ where $n>2$, for which this strategy will not prove that finitely generated subgroups of $\SLn(\ZZ)$ are generically thin.

\subsection{Ping-Pong}\label{pingpongsection}
$$$$

\vspace{-0.2in}

A natural analog in higher rank of the methods used in Section~\ref{sl2sec} to prove generic thinness is the ping-pong argument for subgroups of $\SLn(\mathbb R)$ which is examined in \cite{BG} (note that \cite{BG} also considers $\SLn$ over nonarchimedian fields).  We recall the relevant results here.

We no longer have that $\SLn(\mathbb R)$ acts on $\mathbb H^n$ in the nice way that $\SLt(\mathbb R)$ acts on $\mathbb H$, and so  we consider the action of $\SLn(\mathbb R)$ on real projective space $\mathbb P^{n-1}(\mathbb R)$, viewed as an $n$-dimensional vector space.  We define the distance in $\mathbb P^{n-1}(\mathbb R)$ by
\begin{equation}\label{projdist}
d([v],[w])=\frac{||v\wedge w||}{||v||\cdot||w||},
\end{equation}
where $[v]$ denotes the line spanned by $v$, and $||v\wedge w||$ is defined as follows: writing
$$v\wedge w=\sum_{1\leq i<j\leq n}(v_iw_j-v_jw_i)e_i\wedge e_j,$$
where $(e_1,\dots,e_n)$ is the canonical basis of $\mathbb R^n$, we have
$$||v\wedge w||^2=\sum_{1\leq i<j\leq n}(v_iw_j-v_jw_i)^2.$$
Note that this is simply the sine of the angle between $v$ and $w$.  The metric in (\ref{projdist}) satisfies some useful properties which are described in \cite{BG} and repeated below:
\begin{itemize}
\item Given a linear form $f:\mathbb R^n\rightarrow \mathbb R$, we have
\begin{equation}\label{metricproperty1}
d([v],[\ker f])=\frac{|f(v)|}{||f||\cdot ||v||}
\end{equation}
\item The group $K=\textrm{SO}_n(\mathbb R)$ acts isometrically on $\mathbb P^{n-1}(\mathbb R)$ with the metric $d$.
\end{itemize}

In what follows, we consider the Cartan decomposition of $\SLn(\mathbb R)$, namely that it decomposes as
\begin{equation*}
\SLn(\mathbb R)=K\,A\,K
\end{equation*}
where $K=\textrm{SO}_n(\mathbb R)$ is the maximal compact subgroup of $\textrm{SL}_n(\mathbb R)$ and $$A=\{{\mbox{diag}}(e^{j_1},\dots,e^{j_n})\; |\; j_i\in\mathbb R, j_i\leq j_{i+1},\sum_i j_i=0\}.$$  Specifically, any element $g\in\SLn(\ZZ)$ can be written as
\begin{equation}\label{KAK}
g=k_ga_gk_g'
\end{equation}
where $k_g,k'_g\in K$ and $a_g\in A$.  The matrix $a_g$ is uniquely determined by $g$: it is the diagonal matrix with the eigenvalues of $g^t g$ on the diagonal, ordered from highest to lowest.

In \cite{BG}, Breuillard-Gelander give a way to construct finitely generated free subgroups of $\SLn(\mathbb R)$ (and $\SLn$ over other fields) by producing generators which form a ping-pong $n$-tuple and which hence generate a free group.  Note that a free subgroup of $\SLn(\ZZ)$ is infinite index in $\SLn(\ZZ)$ if $n>2$: one argument proving this is that free groups do not have Kazhdan property $T$ but finite index subgroups of $\SLn(\ZZ)$ do.  Hence it is natural to try to use the characterization of freeness in \cite{BG} to show that the generic finitely generated subgroup of $\SLn(\ZZ)$ is free, and immediately conclude that thinness is also generic.  We recall the definition of a ping-pong pair from \cite{BG} below.

\begin{defn}
Two elements $g_1,g_2\in\SLn(\mathbb R)$ are called a ping-pong pair if both $g_1$ and $g_2$ are $(r,\epsilon)$-very proximal with respect to some $r>2\epsilon>0$, and if the attracting points of $g_i$ and $g_i^{-1}$ are at least distance $r$ apart from the repulsive hyperplanes of $g_j$ and $g_j^{-1}$ in $\mathbb P^{n-1}(\mathbb R)$, where $i\not=j$.
\end{defn}

In the above definition, an element $\gamma\in \SLn(\mathbb R)$ is said to be $(r,\epsilon)$-very proximal if both $\gamma$ and $\gamma^{-1}$ are $(r,\epsilon)$-proximal.  Namely, both $\gamma$ and $\gamma^{-1}$ are $\epsilon$-contracting with respect to some attracting point $v_{\gamma}\in \mathbb P^{n-1}(\mathbb R)$ and some repulsive hyperplane $H_{\gamma}$, such that $d(v_{\gamma}, H_{\gamma}) \geq r$.  Finally, $\gamma$ is called $\epsilon$-contracting if if there exists a point $v_{\gamma}\in \mathbb P^{n-1}(\mathbb R)$ and a projective hyperplane $H_{\gamma}$, such that $\gamma$ maps the complement of the $\epsilon$-neighborhood of $H_{\gamma}$ into the $\epsilon$-ball around $v_{\gamma}$.

One can hope to then prove generic thinness in $\SLn(\ZZ)$ by proving that two elements chosen uniformly at random out of a ball in $\SLn(\ZZ)$ will be $(r,\epsilon)$-very proximal, and that their corresponding attracting points and repulsive hyperplanes are at least $r$ apart with probability tending to $1$ as the radius of the ball grows to infinity.

By Proposition~3.1 in \cite{BG}, the necessary and sufficient condition for $\gamma$ to be $\epsilon$-contracting can be stated simply in terms of the top two singular values of $\gamma$:

\begin{thm}[Proposition~3.1 \cite{BG}]\label{gap}
Let $\epsilon<1/4$ and let $\gamma\in\mathrm{SL}_n(\mathbb R)$.  Let $a_1(\gamma)$ and $a_2(\gamma)$ be the largest and second-largest singular values of $\gamma$, respectively (i.e. largest and second-largest eigenvalues of $\gamma^t\gamma$).  If $\frac{a_2(\gamma)}{a_1(\gamma)} \leq \epsilon^2$, then $\gamma$ is $\epsilon$-contracting. More precisely, writing $\gamma = k_{\gamma}a_{\gamma}k_{\gamma}'$ , one can take $H_{\gamma}$ to be the projective hyperplane spanned by $\{k_{\gamma}'^{-1}(e_i)\}_{i=2}^n$, and $v_{\gamma}= k_{\gamma}(e_1)$, where $(e_1,\dots,e_n)$ is the canonical basis of $\mathbb R^n$.

Conversely, suppose $\gamma$ is $\epsilon$-contracting.  Then $\frac{a_2(\gamma)}{a_1(\gamma)}\leq4\epsilon^2$.
\end{thm}

Hence, if one could show that an element chosen uniformly at random out of a ball in $\SLn(\ZZ)$, as well as its inverse, is expected to have a ``large" ratio between the second-largest and largest singular value, then $(r,\epsilon)$-proximality would follow by appealing to equidistribution in sectors in $\SLn(\mathbb R)$.  One can also use the work of \cite{BG} to prove that two given elements form a ping-pong pair when viewed as acting on $\mathbb P(\bigwedge^k(\mathbb R^n))$ , where $\bigwedge^k$ denotes the exterior power, for a suitable $k$.  To do this, one would in particular need each of the two elements and their inverses to have a large ratio between the $k$-th and $(k+1)$st singular value. Interestingly, as we show in the next section, none of these properties of singular values are generic in the usual Euclidean ball model if $n>2$.  We are, however, able to show that the middle two singular values have large ratio with probability tending to $1$ in a modified Euclidean model, and with this we are able to prove statements on generic thinness in Section~\ref{modified}.

Note, however, that at this point we can already show a version of R.~Aoun's Theorem:
\begin{thm}
\label{aounthm}
Given two long random products $w_1, w_2$ of generators of a Zariski dense subgroup $\Gamma$ of $\textrm{SL}_n(\mathbb{Z}),$ $w_1$ and $w_2$ generate a free subgroup.
\end{thm}
\begin{proof}
By the results of Guivarc'h and Raugi \cite{GR} (see also Goldsheid-Margulis \cite{goldmar}), the assumption of Zariski density implies that the Lyapunov exponents of $\Gamma$ with the given generating set are distinct, and so for the ratio of the top singular value to the second biggest grows exponentially fast as a function of $n.$ Since the words are also known to be equi-distributed in sectors (see, e.g. \cite{bougerol}), the result follows.
\end{proof}

\subsection{$\epsilon$-contraction in $\SLn(\ZZ)$ where $n>2$}\label{epscontract}
$$$$

\vspace{-0.4in}

In this section we prove the following theorem, which from the discussion in the previous section implies in particular that, if $n>2$, it is not true that two elements chosen uniformly at random out of a Euclidean ball of radius $X$ will generically form a ping-pong pair in $\mathbb P^{n-1}(\mathbb R)$ (and hence one cannot conclude that the group generated by two such elements is generically thin via this route). Furthermore, a similar result will hold when one considers $\mathbb P(\bigwedge^k(\mathbb R^n))$ for various $k$, so even with this strategy, the best that one can prove via this method is that a randomly chosen finitely generated subgroup of $\textrm{SL}_n(\mathbb Z)$ will be thin with some \textbf{positive} probability (unfortunately, the probability given by our argument decreases as a function of $n$.)  Instead of comparing the number of elements in $\textrm{SL}_n(\mathbb Z)$ in a ball which are $\epsilon-contracting$ to the total number of elements in the ball, we compare the measures of the analogous sets in $\textrm{SL}_n(\mathbb R)$.  This is essentially identical to the comparison over $\mathbb Z$ by Theorem~1.4 of \cite{EM} after one proves that the sets
\begin{equation}\label{usualballgap}
C_{X,\eta}:=\{{\mbox{diag}}(\alpha_1,\dots, \alpha_n)\; |\; \alpha_i\in\mathbb R, X\geq\alpha_1\geq\eta\alpha_2;\; \alpha_2\geq\alpha_3\geq\cdots\geq \alpha_{n},\prod_i \alpha_i=1 \}.
\end{equation}
where $\eta>16$ is fixed make up a well-rounded sequence of sets in the sense of Definition~\ref{wellroundeddefn}.

\begin{thm}\label{toptwoth}
For $n\geq 3$ fixed, let $G=\mathrm{SL}_n(\mathbb R)$, and let $\mu$ be Haar measure on $\mathrm{SL}_n(\mathbb R)$.  For $g\in G$, denote by $a_1(g),a_2(g),\dots,a_n(g)$ the nonzero entries of the diagonal matrix $a_g$ in the $KAK$ decomposition of $g$ in (\ref{KAK}) with $a_1(g)\geq a_2(g)\geq\cdots\geq a_n(g)>0$.  Fix $\eta>4$. Then
$$0<\frac{\lim_{X\rightarrow\infty} \, \mu(\{g\in G\, | \, ||g||\leq X, a_1(g)/a_2(g)\geq\eta^2\})}{\lim_{X\rightarrow\infty} \, \mu(\{g\in G\, | \, ||g||\leq X\})}<\frac{1}{\eta^{2(n^2-3n+2)}}.$$
\end{thm}

\begin{proof}
Recall that for any $g\in \SLn(\mathbb R)$, we have that $g=k_ga_gk_g'$, where $a_g\in A=\{{\mbox{diag}}(e^{j_1},\dots,e^{j_n})\; |\; j_i\in\mathbb R, j_i\leq j_{i+1},\sum_i j_i=0\}$ and $k_g,k_g'\in \textrm{SO}_n(\mathbb R)$.

Let $H\in A$ be the $n$-tuple $(j_1,\dots,j_n)$.  Consider the Haar measure on $A$ which is
\begin{equation}\label{haar}
\left(\prod_{\lambda\in\Sigma^+}\sinh(\lambda(H))\right)dH
\end{equation}
where $\Sigma^+=\{j_i-j_k\; |\; 1\leq k<i\leq n\}$ is the set of positive roots (see p. 142 in \cite{knapp}).  We consider all elements $\gamma\in\textrm{SL}_n(\mathbb R)$ such that $||\gamma||^2=\lambda_{\max}(\gamma^t\gamma)\leq X^2$: in other words, in the above notation, consider the region $R$ defined by
\begin{eqnarray}\label{polyR1}
e^{j_1}\leq X, && j_1\geq j_2\geq \cdots\geq j_n,\nonumber\\
e^{j_2}\leq X, && j_1+\cdots+j_n=0.\\
\vdots &&\nonumber\\
e^{j_n}\leq X, &&\nonumber
\end{eqnarray}
In order to prove Theorem~\ref{toptwoth}, we examine the ratio between the integral of the expression in (\ref{haar}) over the region $R$ and the integral over the subset of this region where $j_1-j_2\geq T$ and let $X$ tend to $\infty$.  The first integral in question is
$$\int_{R} \left(\prod_{\lambda\in\Sigma^+}\sinh(\lambda(H))\right)dH$$
We expand this as a sum of integrals of exponentials, and the integral becomes
\begin{equation}\label{firstinteven}
C\cdot\int_0^{\log X}\int_{\frac{\minus j_1}{n\minus1}}^{j_1}\int_{\frac{\minus j_1\minus j_2}{n\minus2}}^{j_2}\cdots\int_{\frac{\minus j_1\minus\cdots\minus j_{n\minus2}}{2}}^{j_{n\minus2}}\sum_{\sigma\in S_n} \textrm{sgn}(\sigma) e^{\sum_{k=1}^{n/2}(n-2k+1)(j_{\sigma(k)}-j_{\sigma(n/2+k)})}dj_{n-1}\cdots dj_1
\end{equation}
for $n$ even, or, for $n$ odd,
\begin{equation}\label{firstintodd}
C\cdot\int_0^{\log X}\int_{\frac{\minus j_1}{n\minus1}}^{j_1}\int_{\frac{\minus j_1\minus j_2}{n\minus2}}^{j_2}\cdots\int_{\frac{\minus j_1\minus\cdots\minus j_{n\minus2}}{2}}^{j_{n\minus2}}\sum_{\sigma\in S_n} \minus\textrm{sgn}(\sigma)e^{\sum_{k=1}^{n/2}(n-2k+1)(j_{\sigma(k)}-j_{\sigma((n-1)/2+k)})}dj_{n-1}\cdots dj_1
\end{equation}
where $C=1/2^{\frac{n(n-1)}{2}}$.  Substituting $j_n=-j_1-j_2-\cdots-j_{n-1}$, it is clear that the maximum on $R$ of any given exponential in the sum above is at most $(n^2-n)\log X$, reached always at the point $(\log X,\log X,\dots,\log X,-(n-1)\log X)$. In fact, from \cite{DRS} the integrals above are asymptotic to $cX^{n^2-n}$ for some nonzero constant $c$.  

The second integral, over the intersection of $R$ with the subset of $R$ where $j_1-j_2\geq T$ is the same as above, except that the upper limit of $j_2$ is replaced by $j_1-T$.  Since these two integrals (one over $R$ and the other over a subregion of $R$) differ only in the range of $j_2$, and noting that asymptotically only the terms $e^{s_1j_1+\cdots s_{n-1}j_{n-1}}$ for which $\sum s_i=n^2-n$ contribute, we write 
\begin{eqnarray*}
\int_{R} \left(\prod_{\lambda\in\Sigma^+}\sinh(\lambda(H))\right)dH&\sim& C\cdot\int_0^{\log X}\int_{\frac{\minus j_1}{n\minus1}}^{j_1}\sum_{\substack{i+k=n^2-n\\ 2\leq i\leq 2n-2\\ i\in2\mathbb Z}} a_i(n)e^{ij_1+kj_2}dj_2dj_1\\
&\sim& \sum_{\substack{i+k=n^2-n\\ 2\leq i\leq 2n-2\\ i\in2\mathbb Z}} \frac{a_i(n)}{k(i+k)}e^{(n^2-n)\log X}\\
&=& \alpha X^{n^2-n}
\end{eqnarray*}
for some $\alpha>0$, and 
\begin{eqnarray*}
\int_{R\cap\{j_1-j_2\geq T\}} \left(\prod_{\lambda\in\Sigma^+}\sinh(\lambda(H))\right)dH&\sim& C\cdot\int_0^{\log X}\int_{\frac{\minus j_1}{n\minus1}}^{j_1-T}\sum_{\substack{i+k=n^2-n\\ 2\leq i\leq 2n-2\\ i\in 2\mathbb Z}} a_i(n)e^{ij_1+kj_2}dj_2dj_1\\
&\sim& \sum_{\substack{i+k=n^2-n\\ 2\leq i\leq 2n-2\\ i\in2\mathbb Z}} \frac{a_i(n)}{k(i+k)}e^{(n^2-n)\log X}\cdot e^{-kT}
\end{eqnarray*}
Therefore, since $k\leq n^2-n-2n+2=n^2-3n+2$, we have
$$\int_{R\cap\{j_1-j_2\geq T\}} \left(\prod_{\lambda\in\Sigma^+}\sinh(\lambda(H))\right)dH\sim \beta X^{n^2-n}$$
where 
$$0<\beta<\alpha e^{-(n^2-3n+2)T}$$
which proves the claim.
\end{proof}

\subsection{Proof of Theorem~\ref{thinsln}}\label{modified}
$$$$

\vspace{-0.2in}

Although we have shown in the previous section that two elements chosen uniformly at random from a Euclidean ball in $\SLn(\ZZ)$ are not expected to form a ping-pong pair, we show in this section that two elements chosen uniformly at random from a rather natural modification of the notion of Euclidean ball in $G:=\textrm{SL}_n(\ZZ)$ will form a ping-pong pair.  Let
$$B_X'(G):=\{\gamma\in G\; |\; ||\gamma||<X \mbox{ and } ||\gamma^{-1}||<X\},$$
and let $\mu_X'$ denote the normalized counting measure on $(B_X')^2$ with respect to this region.  We then have the following.
\begin{prop}\label{freesln}
Let $G=\mathrm{SL}_n(\ZZ)$, and let $\Gamma(g)$ and $\mu_X'$ be as above.  Then we have
\begin{equation*}
\lim_{X\rightarrow\infty}\mu'_X(\{g=(g_1,g_2)\in G^2\; |\; \Gamma(g) \mbox{\emph{ is free}}\})=1
\end{equation*}
\end{prop}

As we note below, this is almost Theorem~\ref{thinsln}.  To prove this proposition, we show that generically our generators will form a ping-pong pair in some suitable space $\mathbb P(\bigwedge^k(\mathbb R^n))$.  According to \cite{BG}, the first step is to show that they are generically $\epsilon$-very contracting.

\begin{lemma}\label{middletwo}
Let $G=\mathrm{SL}_n(\ZZ)$, and let $\mu'_X$ be as above.  For $g\in G$, denote by $a_1(g),\dots, a_n(g)$ the nonzero entries of the diagonal matrix $a_g$ in the $KAK$ decomposition of $g$ in (\ref{KAK}) with $a_1(g)\geq a_2(g)\geq \cdots\geq a_n(g)>0$.  Fix $\eta>4$. Then, if $n$ is even, we have 
$$\lim_{X\rightarrow\infty} \, \mu'_X\left(\left\{g\in G\, \Big| \, \frac{a_k(g)}{a_{k+1}(g)}\geq\eta^2 \mbox{\emph{ and }} \frac{a_k(g^{-1})}{a_{k+1}(g^{-1})}\geq\eta^2\right\}\right)=1,$$
where $k=n/2$.  If $n$ is odd, we have that 
$$\lim_{X\rightarrow\infty} \, \mu'_X\left(\left\{g\in G\, \Big| \, \frac{a_k(g)}{a_{k+1}(g)}, \frac{a_{k+1}(g)}{a_{k+2}(g)}\geq\eta^2, \mbox{\emph{ and }} \frac{a_k(g^{-1})}{a_{k+1}(g^{-1})},  \frac{a_{k+1}(g^{-1})}{a_{k+2}(g^{-1})}\geq\eta^2\right\}\right)=1,$$
where $k=(n-1)/2$.
\end{lemma}

Note that, given the definition of $k$ in the two cases in Lemma~\ref{middletwo}, whenever the mentioned ratios between singular values of $g$ are large enough, so are the ratios between the relevant pairs of the singular values of $g^{-1}$.  Hence we need only to prove the statements above for singular values of $g$.

\begin{proof}
Let $T=2\log\eta$ and consider, as $X\rightarrow\infty$, the ratio
$$\frac{|\{g\in B'_X(G)\;|\; j_k(g)-j_{k+1}(g)\geq T, \}|}{|\{g\in B'_X(G)\}|}$$
if $n=2k$ is even, and 
$$ \frac{|\{g\in B'_X(G)\;|\; j_k(g)-j_{k+1}(g)\geq T, j_{k+1}(g)-j_{k+2}(g)\geq T\}|}{|\{g\in B'_X(G)\}|}$$
if $n=2k+1$ is odd, where $j_i(g)$ is defined as in the previous section.  We replace now $G$ by $\textrm{SL}_n(\mathbb R)$, and the norm above with Haar measure, noting that, as in the proof of Theorem~\ref{toptwoth}, the ratio we get in this way will be asymptotic to the ratio above by Theorem~1.4 in \cite{EM} (see Theorem~\ref{emtheorem}) along with Lemmata~\ref{wellrounded1} and \ref{gapwellrounded}.

We consider all elements $\gamma\in\textrm{SL}_n(\mathbb R)$ such that $||\gamma||^2=\lambda_{\max}(\gamma^t\gamma)\leq X^2$ and $||\gamma^{-1}||\leq X^2$: in other words, in the above notation, consider the region $R$ is the convex polygon defined by
\begin{eqnarray}\label{polyR2}
e^{|j_1|}\leq X, && j_1\geq j_2\geq \cdots\geq j_n,\nonumber\\
e^{|j_2|}\leq X, && j_1+\cdots+j_n=0.\\
\vdots &&\nonumber\\
e^{|j_n|}\leq X, &&\nonumber
\end{eqnarray}
The first integral in question is
$$\int_{R} \left(\prod_{\lambda\in\Sigma^+}\sinh(\lambda(H))\right)dH$$
Again we expand this as a sum of integrals of exponentials.  If $n$ is even, we obtain
\begin{equation}\label{weirdfirstinteven}
C\cdot\int_{R}\sum_{\sigma\in S_n} \textrm{sgn}(\sigma) e^{\sum_{m=1}^{n/2}(n-2m+1)(j_{\sigma(m)}-j_{\sigma(n/2+m)})}dj_{n-1}\cdots dj_1
\end{equation}
for $n$ even, or, for $n$ odd,
\begin{equation}\label{weirdfirstintodd}
C\cdot\int_{R}\sum_{\sigma\in S_n} \minus\textrm{sgn}(\sigma)e^{\sum_{m=1}^{(n-1)/2}(n-2m+1)(j_{\sigma(m)}-j_{\sigma((n-1)/2+m)})}dj_{n-1}\cdots dj_1
\end{equation}
where $C=1/2^{\frac{n(n-1)}{2}}$.  Note that on $R$, the maximum of any one of the exponentials in the sums above is, depending on $\sigma$, is
\begin{equation}\label{maxeven}
X^{2n-2+2n-4+\cdots+2}
\end{equation}
or of smaller order both for $n$ even and odd.  When this maximum is achieved (i.e. when one considers an appropriate $\sigma$ in the sum above), it is achieved at the point $P=(\log X,\log X,\dots,\log X,-\log X, -\log X,\dots, -\log X)$ in the even case, and at $Q=(\log X,\log X,\dots,\log X,0,-\log X,-\log X,\dots, -\log X)$ in the odd case.  Since the exponentials are all exponentials of linear functions, and since $P$ and $Q$ are contained in $R$ (in the even and odd cases, respectively), we have that the maximum obtained on $R$ for any one of the exponentials in the sums above is precisely the expression in (\ref{maxeven}).

We now separate the two cases $n$ even and odd.

\noindent{\it Case 1:} $n=2k>2$ is even.

Let $f(j_1,\dots,j_{n-1})$ be the constant multiple of the sum of exponentials in (\ref{weirdfirstinteven}).  We will show that 

\begin{equation}\label{wholeevenasym}
\int_{R}f(j_1,\dots,j_{n-1})dj_{n-1}\dots dj_1 \sim \alpha X^{2(n-1) +2(n-3)+\cdots+2}=\alpha X^{n^2/2}
\end{equation}
for some constant $\alpha$.  First, note that $f(j_1,\dots,j_{n-1})$ is nonnegative on $R$, and hence an integral of $f$ over a subregion of $R$ will give a lower bound on this integral.  Consider the subregion $R'$ obtained by taking the intersection of $R$ with the region
$$j_{i-1}\geq j_i+M  \mbox{ for all $2\leq i \leq n$}$$
where $M=\log(n!+1)$.  Now, any exponential in the sum of exponentials defining $f$ has a maximum of order less than $cX^{n^2/2}$ where $c>0$ unless it is of the form
$$e^{(n-1)j_{\sigma(1)}+(n-3)j_{\sigma(2)}+\cdots+j_{\sigma(n/2)}-j_{\tau(n/2+1)}-3j_{\tau(n/2+2)}-\cdots-(n-1)j_{\tau(n)}}$$
where $\sigma,\tau\in S_{n/2}$, and hence we may replace $f$ in our computation by 
$$g(j_1,\dots, j_{n-1}):=\sum_{\sigma,\tau\in S_{n/2}}(\sgn(\sigma\tau) e^{(n-1)j_{\sigma(1)}+(n-3)j_{\sigma(2)}+\cdots+j_{\sigma(n/2)}-j_{\tau(n/2+1)}-3j_{\tau(n/2+2)}-\cdots-(n-1)j_{\tau(n)}}).$$
The region $R'$ was defined in such a way that $g$ is at least 
$$c\cdot e^{(n-1)j_{1}+(n-3)j_{2}+\cdots+j_{n/2}-j_{n/2+1}-3j_{n/2+2}-\cdots-(n-1)j_{n}}$$
on $R'$, for some positive constant $c.$ We hence show that the integral
\begin{equation}\label{reduceoneeven}
\int_{R'} c\cdot e^{(n-1)j_{1}+(n-3)j_{2}+\cdots+j_{n/2}-j_{n/2+1}-3j_{n/2+2}-\cdots-(n-1)j_{n}}dj_{n-1}\cdots dj_1\gg X^{n^2/2},
\end{equation}
thus proving the asymptotic in (\ref{wholeevenasym}), since the maximum of $f$ over $R$ is of order at most $X^{n^2/2}$ as discussed above.  To prove (\ref{reduceoneeven}), note that the exponential in the integral in $(\ref{reduceoneeven})$ is positive everywhere, and so it is bounded below by 
\begin{equation}\label{reduceoneeven2}
\int_{R'\cap B_{\epsilon'}} c\cdot e^{(n-1)j_{1}+(n-3)j_{2}+\cdots+j_{n/2}-j_{n/2+1}-3j_{n/2+2}-\cdots-(n-1)j_{n}}dj_{n-1}\cdots dj_1\gg X^{n^2/2},
\end{equation}
where $\epsilon'>0$ and 
$$B_{\epsilon'}:=[\log X-\epsilon',\log X]^{n/2}\times[-\log X,-\log X+\epsilon']^{n/2}.$$
Since the minimum of the exponential above over $R'\cap B_{\epsilon'}$ is easily seen to be $c' X^{n^2/2}$ for a constant $c'>0$ depending on $\epsilon'$ and $M$, we have that the integral in (\ref{reduceoneeven2}) is bounded below by $V(R'\cap B_{\epsilon'})\cdot c' X^{n^2/2}$, where $V(R'\cap B_{\epsilon'})$ denotes the area of $R'\cap B_{\epsilon'}$ which is at least a constant depending on $\epsilon'$ and $M$.  Thus we have that (\ref{reduceoneeven}) holds and so the asymptotic in (\ref{wholeevenasym}) is correct.

Let $T>0$ be fixed, and let $R_T$ denote the region defined by $j_{n/2}\leq j_{n/2+1}+T$.  We now show that
$$\int_{R\cap R_T} f(j_1,\dots,j_{n-1})dj_{n-1}\cdots dj_1$$
is of lower order than $X^{n^2/2}$, proving Lemma~\ref{middletwo} for even $n$.

To do this, consider any one of the exponentials in the sum defining $f$.  It is of the form
$$e^{(n-1)j_{\sigma(1)}+(n-3)j_{\sigma(2)}+\cdots+j_{\sigma(n/2)}-j_{\sigma(n/2+1)}-3j_{\sigma(n/2+2)}-\cdots-(n-1)j_{\sigma(n)}}$$
where $\sigma\in S_n$.  In the case that the coefficients of the $j_{n/2}$ and $j_{n/2+1}$ terms have opposite sign, the maximum value over $R\cap B_T$ of $\sum_{i\not=n/2,n/2+1} a_ij_i$ in the exponent is bounded above by $(\frac{n^2}{2}-a_{n/2}+a_{n/2+1})\log X$ if $a_{n/2+1}$ is negative, or $(\frac{n^2}{2}+a_{n/2}-a_{n/2+1})\log(X)$ if $a_{n/2}$ is negative.  Now consider the remaining part of the exponent,
\begin{equation}\label{gappedtwo}
a_{n/2}j_{n/2}+a_{n/2+1}j_{n/2+1}\leq (a_{n/2}+a_{n/2+1})j_{n/2}-(a_{n/2}+a_{n/2+1})T\leq (a_{n/2}+a_{n/2+1})j_{n/2}.
\end{equation}
Adding this to the upper bound on the maximum value of the rest of the terms in the exponent, we get an upper bound of
$$\left(\frac{n^2}{2}-1\right)\log X$$
for the exponent, and hence an upper bound of $X^{n^2/2-1}$ for the corresponding exponential in the sum defining $f$.  If the coefficients of the $j_{n/2}$ and $j_{n/2+1}$ terms have the same sign, then we get that the maximum value over $R\cap B_T$ of $\sum_{i\not=n/2,n/2+1} a_ij_i$ in the exponent is bounded above by $(\frac{n^2}{2}-a_{n/2}-a_{n/2+1}-2)\log X$, and adding the remaining part of the exponent in (\ref{gappedtwo}) we get an upper bound of
$$\left(\frac{n^2}{2}-2\right)\log X$$
for the exponent, and hence an upper bound of $X^{n^2/2-1}$ for the corresponding exponential in the sum defining $f$.  Hence we have that 
\begin{equation}\label{uppergap}
\int_{R\cap R_T} f(j_1,\dots,j_{n-1})dj_{n-1}\cdots dj_1\ll X^{n^2/2-1},
\end{equation}
which is of lower order than the asymptotic for the integral over $R$ in (\ref{wholeevenasym}).  This concludes the proof of the lemma for even $n$.

\noindent{\it Case 2:} $n=2k+1\geq 3$ is odd.

This case is extremely similar to the even case above, so we suppress all of the details, but note the key differences.  In this case, the maximum over $R$ of any given exponential in the sum defining $f$ is $X^{n(n-1)/2}$, obtained at the point $(\log X,\log X,\dots,\log X,0,-\log X,\newline -\log X,\dots,-\log X)$.  The problem of finding the asymptotic for the integral of $f$ over $R$ reduces to obtaining a suitable lower bound on the asymptotic for
\begin{equation}\label{reduceoneodd}
\int_{R'} c\cdot e^{(n-1)j_{1}+(n-3)j_{2}+\cdots+2j_{k}-2j_{k+2}-4j_{n/2+2}-\cdots-(n-1)j_{n}}dj_{n-1}\cdots dj_1
\end{equation}
over a region $R'$ defined very similarly as the one in the even case.  This lower bound is obtained by considering the integral over $R\cap B_{\epsilon'}$ where $B_{\epsilon'}$ is defined to be the region
$$B_{\epsilon'}:=[\log X-\epsilon',\log X]^{n/2}\times[-\epsilon',\epsilon']\times[-\log X,-\log X+\epsilon']^{n/2}.$$

Next, for $T_1,T_2>0$ we define $R_{T_1,T_2}$ to be the region defined by $j_k\geq j_{k+1}+T_1, j_{k+1}\geq j_{k+2}+T_2$, which is the analogue of $R_T$ in Case 1.  The proof that 
$$\int_{R\cap R_{T_1,T_2}} f(j_1,\dots,j_{n-1})dj_{n-1}\cdots dj_1\ll X^{n^2-2\log X}$$
is almost identical to the proof of (\ref{uppergap}) above, and hence we have that the statement in the lemma holds for odd $n$ as well.

\end{proof}

The spectral gap from the previous lemma gives us that our generators will be $\epsilon$-contracting with probability tending to $1$ in a suitable space.  Specifically, consider $v,w\in \mathbb P(\bigwedge^k\mathbb R^n)$.  We define a metric

\begin{equation}\label{projdistexterior}
d([v],[w])=\frac{||v\wedge w||}{||v||\cdot||w||},
\end{equation}
where $[v]$ denotes the line spanned by $v$, and $||v||$, $||w||$, and $||v\wedge w||$ are defined in a canonical way after fixing a basis
$$B:=\{e_{i_1}\wedge\cdots\wedge e_{i_k}\;|\; 1\leq i_1<i_2<\cdots<i_k\leq n\}$$
for $\bigwedge^k(\mathbb R^n)$.  For example, writing
$$v=\sum_{1\leq i_1<i_2<\cdots<i_k\leq n}v_{i_1,\dots,i_k}e_{i_1}\wedge\cdots\wedge e_{i_k},$$
we have
$$||v||^2=\sum_{1\leq i_1<i_2<\cdots<i_k\leq n}v_{i_1,\dots,i_k}^2.$$
In addition, we define an action of $K=\textrm{SO}_n(\mathbb R)$ on $\mathbb P(\bigwedge^k(\mathbb R^n))$ by considering its action on $\mathbb R^n$ with basis $\{e_1,\dots,e_n\}$.  The metric in (\ref{projdistexterior}) satisfies the following properties:
\begin{itemize}
\item Given a linear form $f:\mathbb \bigwedge^k(\mathbb R^n)\rightarrow \mathbb R$, we have
\begin{equation}\label{metricproperty1exterior}
d([v],[\ker f])=\frac{|f(v)|}{||f||\cdot ||v||}
\end{equation}
\item The group $K=\textrm{SO}_n(\mathbb R)$ acts isometrically on $\mathbb P(\bigwedge^k(\mathbb R^n))$ with the metric $d$.
\end{itemize}

\begin{lemma}\label{epscontractextpower}
Let $\epsilon>0$ and let $\gamma\in\mathrm{SL}_n(\mathbb R)$.  Let $a_i(\gamma)$ denote the $i$-th largest singular value of $\gamma$.  If $\frac{a_{k+1}(\gamma)}{a_{k}(\gamma)} \leq \epsilon^2$, then $\gamma$ is $\epsilon$-contracting when acting on $\mathbb P(\bigwedge^{k}(\mathbb R^n))$.

More precisely, let $B:=\{e_{i_1}\wedge\cdots\wedge e_{i_k}\;|\; 1\leq i_1<i_2<\cdots<i_k\leq n\}$, a basis for $\bigwedge^k(\mathbb R^n)$ (here $(e_1,\dots,e_n)$ is the canonical basis of $\mathbb R^n$).  Then, writing $\gamma = k_{\gamma}a_{\gamma}k_{\gamma}'$ , one can take $H_{\gamma}$ to be the projective hyperplane spanned by $\{k_{\gamma}'^{-1}(\mathbf{v})\;|\; \mathbf v\in B\backslash e_1\wedge e_2\wedge\cdots\wedge e_k\}$, and $v_{\gamma}= k_{\gamma}(e_1\wedge e_2\wedge\cdots\wedge e_k)$.  Then we have that $\gamma$ maps the outside of the $\epsilon$-neighborhood of $H_{\gamma}$ into the $\epsilon$-ball around $v_{\gamma}$.

\end{lemma}

\begin{proof}
The proof of this is a straightforward generalization of the proof of Theorem~\ref{gap}, or Proposition~3.1 in \cite{BG}.  Suppose that $\frac{a_{k+1}(\gamma)}{a_{k}(\gamma)} \leq \epsilon^2$.  Since $K$ acts by isometries on $\mathbb P(\bigwedge^k(\mathbb R^n))$, we may reduce to the diagonal case,
$$\gamma=a_{\gamma}=\textrm{diag}(a_1,\dots,a_n).$$
Suppose $[v]$ is outside the $\epsilon$-neighborhood of $H_{\gamma}$: i.e., by (\ref{metricproperty1exterior}), suppose
$$d([v],H_{\gamma})=\frac{|v_{1,2,\dots,k}|}{||v||}\geq\epsilon,$$
where we write
$$v=\sum_{1\leq i_1<i_2<\cdots<i_k\leq n}v_{i_1,\dots,i_k}e_{i_1}\wedge\cdots\wedge e_{i_k}.$$
Then 
$$d([\gamma v],[e_1\wedge e_2\wedge\cdots\wedge e_k])=\frac{||\gamma v\wedge e_1\wedge e_2\wedge\cdots\wedge e_k||}{||\gamma v||}\leq \frac{a_{k+1}||v||}{a_{k}|v_{1,2,\dots,k}|}\leq \epsilon,$$
since $||\gamma v\wedge e_1\wedge e_2\wedge\cdots\wedge e_k||\leq a_1\cdots a_{k-1}a_{k+1}||v||$, and $||\gamma v||\geq a_1\cdots a_{k}|v_{1,2,\dots,k}|$.
\end{proof}

Lemma~\ref{epscontractextpower} together with Lemma~\ref{middletwo} immediately imply the following:
\begin{lemma}\label{contracts}
Let $G=\mathrm{SL}_n(\ZZ)$, fix $0<\epsilon\leq 1/4$, and let $\mu'_X$ be as above.  Then we have that 
$$\lim_{X\rightarrow\infty} \, \mu'_X(\{g\in G\, | \, g \mbox{ is $\epsilon$-very contracting in $\mathbb P(\bigwedge^k(\mathbb R^n))$})=1,$$
where $k=n/2$ for $n$ even and $k=(n-1)/2$ for $n$ odd. 
\end{lemma}

We now show that generically our chosen generators will be $r$-very proximal.

\begin{lemma}\label{prox}
Let $G=\mathrm{SL}_n(\ZZ)$, where $n>2$, and let $\mu'_X$ be as before.  For $g\in G$, let $g=k_ga_gk_g'$ be a Cartan decomposition of $g$ and let $k=n/2$ if $n$ is even, and $k=(n-1)/2$ if $n$ is odd.  Denote by $H_g$ the projective hyperplane in $\mathbb P(\bigwedge^k(\mathbb R^n))$ spanned by $\{k_{g}'^{-1}(\mathbf{v})\;|\; \mathbf v\in B\backslash e_1\wedge e_2\wedge\cdots\wedge e_k\}$ where $B$ is as in Lemma~\ref{epscontractextpower}, and let $v_g=k_{g}(e_1\wedge e_2\wedge\cdots\wedge e_k)$.  Then, as $\epsilon'$ tends to $0$, we have
$$\lim_{X\rightarrow\infty}\, \mu'_X(\{g\in G\, | \, d(v_g,H_g)\leq 2\epsilon'\})\ll \epsilon'.$$
\end{lemma}

Since we have that for any fixed $0<\epsilon'<1/4$, the probability as $X$ tends to infinity that a randomly chosen $g\in G$ is $\epsilon'$-very
contracting in $\mathbb P(\bigwedge^k(\mathbb R^n))$ tends to $1$, the above gives us that the probability that such an element is $(r,\epsilon')$-very proximal for some $0<\epsilon'<1/4$ also tends to $1$ as $X$ tends to infinity.

\begin{proof}

Let $\epsilon=2\epsilon'$, and for $\kappa\in K$ let $v_{\kappa}:=\kappa(e_1\wedge e_2\wedge\cdots\wedge e_k)$, and let $H_{\kappa}$ denote the hyperplane spanned by $\{\kappa^{-1}(\mathbf{v})\;|\; \mathbf v\in B\backslash e_1\wedge e_2\wedge\cdots\wedge e_k\}$ in $\mathbb P(\bigwedge^k(\mathbb R^n))$.  Let $\Omega\subset K^2$ denote the set of pairs $(\kappa,\kappa')$ such that $d(H_{\kappa'},v_{\kappa})<\epsilon$, where $d$ is the projective distance defined in (\ref{projdistexterior}).  Let $S_{\Omega, X}:=\{\kappa a\kappa' \;|\; (\kappa,\kappa')\in\Omega, \, ||a||<X, ||a^{-1}||>1/X\}$.  To prove Lemma~\ref{prox}, we use Theorem~1.6 from \cite{GO} to obtain an upper bound on 
$$\#(\textrm{SL}_n(\mathbb Z)\cap S_{\Omega, X}).$$
Note that this theorem of \cite{GO} applies to a well-rounded sequences of growing regions.  Hence we show in Lemma~\ref{gapwellrounded} that the sequence of sets $\{g\in\textrm{SL}_n(\mathbb Z)\; |\; ||g||<X, ||g^{-1}||>1/X, a_g(k+1)/a_g(k)<\epsilon'\}_X$ is indeed well-rounded.  With this in mind, let $\mu$ denote the normalized Haar measure on $K$.    

 Consider the orbit $\mathcal O$ of $K$ acting on $e_1\wedge e_2\wedge\cdots\wedge e_k$.  It is a smooth compact manifold, and we denote its dimension by $d$ (in fact, this orbit is the Grassmanian, whose dimension is $k(n-k)$, but we do not need this).  We cover $\mathcal O$ with $\ell= C_V/(\epsilon')^d$ $\epsilon$-balls $B_i$ of full dimension, i.e. dimension of the whole projective space $\mathbb P(\bigwedge^k(\mathbb R^n))$, such that whenever $p\in \mathcal O$ and $q$ is a point with $d(p,q)<\epsilon$, we have that $p$ and $q$ are both contained in some one ball in this cover. Here $C_V$ is a constant depending on the volume of $\mathcal O$.  In particular, any point $v_{\kappa}$ (defined above) is in $\mathcal O$, and if a hyperplane $H_{\kappa'}$ is less than $\epsilon$ away from $v_\kappa$ then it intersects some ball from this cover which contains $v_{\kappa}$.

Let $S_i:=\{\kappa\in K\; |\; v_{\kappa}\in B_i\}$, and let $S_i':=\{\kappa\in K\; |\; H_{\kappa} \cap B_i \mbox{ is nonempty}\}$.  Then 

$$\textrm{SL}_n(\mathbb Z)\cap S_{\Omega, X}\subset \bigcup_{1\leq i \leq \ell} \{g\in \textrm{SL}_n(\mathbb Z)\;|\; ||g||, ||g^{-1}||<X, \; g=\kappa a \kappa', \kappa\in S_i, \kappa'\in S_i'\},$$

and hence, by Theorem~1.6 from \cite{GO}, we have

$$\frac{\#(\textrm{SL}_n(\mathbb Z)\cap S_{\Omega, X})}{\#\{g\in\textrm{SL}_n(\mathbb Z)\;|\; ||g||, ||g^{-1}||<X\}} \leq  \sum_{1\leq i \leq\ell} \mu(S_i)\mu(S_i').$$

Now, as $\epsilon'\rightarrow 0$, we have that $\mu(S_i)\sim\epsilon'^d$, and $\mu(S_i')\leq \epsilon'$ since the dimension of the orbit under $K$ of any point in the hyperplane $H_{I}$, where $I$ is identity, is at least $1$.  Hence we have that, as $\epsilon'$ tends to $0$,

$$\frac{\#(\textrm{SL}_n(\mathbb Z)\cap S_{\Omega, X})}{\#\{g\in\textrm{SL}_n(\mathbb Z)\;|\; ||g||, ||g^{-1}||<X\}} \ll \epsilon'$$

\end{proof}

In fact, the above argument shows more than what is stated in the Lemma: it also shows that the probability that two elements $g_1,g_2$ chosen uniformly at random out of a modified ball $B_X'$ in $\textrm{SL}_n(\mathbb Z)$ are associated to points $v_{g_1}, v_{g_2}$ and hyperplanes $H_{g_1}, H_{g_2}$ as defined above where any two of these are at most $\epsilon'$ apart tends to $\epsilon'$ as $X\rightarrow\infty$.  Putting this together with  Lemma~\ref{middletwo} and Lemma~\ref{prox}, we have that 
$$\lim_{X\rightarrow\infty}\, \mu_X(\{(g_1,g_2)\in G^2\, | \, \Gamma(g_1,g_2) \mbox{ is free}\})=1$$
where $G=\textrm{SL}_n(\mathbb Z)$ and $n>2$, proving Proposition~\ref{freesln}.  Since free subgroups of $\textrm{SL}_n(\mathbb Z)$ are infinite index in $\textrm{SL}_n(\mathbb Z)$, and since the generic finitely generated subgroup of $\textrm{SL}_n(\mathbb Z)$ is Zariski dense in $\SLn(\mathbb C)$ by \cite{R2}, Theorem~\ref{thinsln} follows.
 
\noindent{\bf{Remark:}} The arguments above are all easily extended to groups generated by $k$ elements, where $k>2$ is fixed.

\section{Well Roundedness}\label{wellroundedsec}

In the preceding section, we use heavily the following result from \cite{EM}.

\begin{thm}[Theorem~1.4, Eskin-Mcmullen \cite{EM}]\label{emtheorem}
Let $G$ be a connected semisimple Lie group with finite center and let $H <G$ be a closed subgroup such that $V=G/H$ is an affine symmetric space.  Let $\Gamma$ be a lattice in $G$, and let $v$ denote the coset $[H]$. For any well-rounded sequence, the cardinality of the number of points of $\Gamma v$ which lie in $B_n$, grows like the volume of $B_n$: asymptotically,
$$|\Gamma v\cap B_n|\sim \frac{m((\Gamma\cap H)\backslash H}{\Gamma\backslash G}m(B_n).$$
\end{thm}

We now prove that the two sequences of growing regions which are featured in the previous section are indeed well-rounded in the following sense (see \cite{EM}):

\begin{defn}\label{wellroundeddefn} Let $\SLn(\mathbb R)=KAK$ be the Cartan decompostition, and let $\{A_i\}_{i\geq1}$ be a sequence of subsets of $A$ such that the volume of $KA_iK$ tends to infinity as $i\rightarrow\infty$.  The sequence of sets $\{KA_iK\}_{i\geq1}$ is \emph{well-rounded} if for all $\epsilon'>0$ there exists a neighborhood $U$ of identity in $\textrm{SL}_n(\mathbb R)$ such that 
\begin{equation}\label{wellrounded}
(1-\epsilon')m\left(\bigcup_{g\in U} gKA_iK\right)< m(KA_iK)< (1+\epsilon')m\left(\bigcap_{g\in U} gKA_iK\right).
\end{equation}
for all $i$.
\end{defn}

Proving well-roundedness is necessary in order to use the following theorem of Eskin-Mcmullen \cite{EM} to pass between computing volumes of regions in $\textrm{SL}_n(\mathbb R)$ and counting points in $\textrm{SL}_n(\mathbb Z)$ throughout the paper.  It is also necessary in using the equidistribution result of Gorodnick-Oh \cite{GO} in proving Lemma~\ref{prox}.  Throughout this section, we assume $n>2$.

Note that for our purposes in Theorem~\ref{emtheorem} to obtain an asymptotic count of the number of points in $|\Gamma v\cap B_n|$, we need only show that $\{B_n\}_n>N$ is a well-rounded sequence for some $N\in\mathbb N$.

For $T\in \mathbb R$, let
$$C_T':=\{{\mbox{diag}}(\alpha_1,\dots, \alpha_n)\; |\; \alpha_i\in\mathbb R, T\geq\alpha_1\geq\cdots\geq \alpha_{n}\geq 1/T,\prod_i \alpha_i=1\},$$
and, fixing $\beta>16$, let
\begin{equation}\label{weirdballgap1}
C_{T,\beta}':=\{{\mbox{diag}}(\alpha_1,\dots, \alpha_n)\; |\; \alpha_i\in\mathbb R, T\geq\alpha_1\geq\cdots\alpha_{n/2}\geq\beta\alpha_{n/2+1}\geq\cdots\geq \alpha_{n}\geq 1/T,\prod_i \alpha_i=1\}\nonumber
\end{equation}
if $n$ is even, and let
\begin{eqnarray}\label{weirdballgap2}
C_{T,\beta}'&:=&\{{\mbox{diag}}(\alpha_1,\dots, \alpha_n)\; |\; \alpha_i\in\mathbb R,\nonumber \\
&&T\geq\alpha_1\geq\cdots\alpha_{(n-1)/2}\geq\beta\alpha_{(n+1)/2}\geq\beta^2\alpha_{(n+3)/2}>\alpha_{(n+3)/2}\geq\cdots\geq \alpha_{n}\geq 1/T,\prod_i \alpha_i=1\}\nonumber
\end{eqnarray}
if $n$ is odd.  In \cite{EM} it is shown that the sequence $\{KC_{T}K\}$ of regions whose volume tends to infinity as $T\rightarrow\infty$ is well-rounded.

We show that this is true for the other sets above as well.  We begin with the sequence $\{KC_{T}'K\}$.

\begin{lemma}\label{wellrounded1}
Let $C_T'$ be defined as above.  Then the sequence $\{KC_{T}'K\}$ of regions (whose volume tends to infinity as $T\rightarrow\infty$) is well-rounded.
\end{lemma}
\begin{proof}

We first show that for any $\epsilon>0$ there is a neighborhood $\mathcal N_{1,\epsilon}$ of identity in $\textrm{SL}_n(\mathbb R)$ such that
$$\bigcup_{g\in \mathcal N_{1,\epsilon}} gKC_T'K\subset KC_{(1+\epsilon)T}'K.$$
Let $\mathcal O_{\sqrt[n-1]{1+\epsilon}}$ be the neighborhood of identity of radius $\sqrt[n-1]{1+\epsilon}$.  Then $\mathcal O_{\sqrt[n-1]{1+\epsilon}}$ contains $U_1V$ for some neighborhood $U_1$ of identity in $K$ and some neighborhood $V$ of identity in $A$.  By the strong wavefront lemma (Theorem~2.1 in \cite{GO}) there exists a neighborhood $\mathcal N_{1,\epsilon}$ of identity in $\textrm{SL}_n(\mathbb R)$ such that $$\mathcal N_{1,\epsilon}\gamma\subset k_1U_1aVk_2=k_1U_1Vak_2$$ for all $\gamma=k_1ak_2$ in $KC_T'K$.  Hence we have
$$\mathcal N_{1,\epsilon} KC_T'K\subset KU_1VC_T'K\subset K \mathcal O_{\sqrt[n-1]{1+\epsilon}}C_T'K\subset KC_{(1+\epsilon)T}'K.$$
The last containment follows from the submultiplicativity of the spectral norm.  Namely, for any $g\in \mathcal O_{\sqrt[n-1]{1+\epsilon}}$ and any $h\in C_T'$, we have that $||gh||\leq ||g||\cdot||h||<T\sqrt[n-1]{1+\epsilon} <(1+\epsilon)T$, and $||(gh)^{-1}||\leq ||h^{-1}||\cdot||g^{-1}||<T(1+\epsilon)$ as well, and so $K\mathcal O_{\sqrt[n-1]{1+\epsilon}} C_T'K\subset KC_{(1+\epsilon)T}'K$.

We also have that for any $0<\epsilon<1$ there exists a neighborhood $\mathcal N_{2,\epsilon}$ of identity in $\textrm{SL}_n(\mathbb R)$ such that
$$KC_{(1-\epsilon)T}'K\subset \bigcap_{g\in \mathcal N_{2,\epsilon}}gKC_T'K.$$
In other words, we have that 
$$g^{-1}KC_{(1-\epsilon)T}'K\subset KC_T'K$$
for every $g\in \mathcal N_{2,\epsilon}$.  To see this, take $\mathcal N_{2,\epsilon}$ to be the neighborhood of radius $\frac{1}{\sqrt[n-1]{1-\epsilon}}$.  Then $g^{-1}$ is of norm less than $\frac{1}{1-\epsilon}$ and, for any $h\in KC_{(1-\epsilon)T}'K$, we have $||g^{-1}h||\leq ||g^{-1}||\cdot||h||<T$, and $||(g^{-1}h)^{-1}||\leq ||h^{-1}||\cdot||g||<T(1-\epsilon)^{1-\frac{1}{n-1}}<T$ for all $\epsilon>0$.  Hence we have $g^{-1}h\in KC_T'K$ for all $g\in \mathcal N_{2,\epsilon}$ and $h\in KC_{(1-\epsilon)T}'K$.

Recall from Section~\ref{modified} that $m(KC_T'K)\sim p(T)$ where $p$ is a polynomial of degree $n^2/2$ if $n$ is even and $n(n-1)/2$ if $n$ is odd.  For any $0<\epsilon'<1$ one can find $\epsilon_1>0$ such that $(1-\epsilon') =p(\frac{1}{1+\epsilon_1}$) and for any $\epsilon'>0$ there is an $0<\epsilon_2<1$ such that $(1+\epsilon') =p(\frac{1}{1-\epsilon_2})$. Let $U=\mathcal N_{1,\epsilon_1}\cap\mathcal N_{2,\epsilon_2}$, and note that (\ref{wellrounded}) is indeed satisfied for this choice of $U$.  

\vspace{0.1in}

\end{proof}

Now, given that $\{KC_T'K\}_{T\geq 1}$ is a well-rounded sequence of sets, we are able to prove the following lemma, using a slightly different definition of well-roundedness, which is equivalent to Definition~\ref{wellroundeddefn} by \cite{EM}.

\begin{defn}\label{wellroundeddefn2}
The sequence $\{B_n\}$ of sets is \emph{well-rounded} if for any $\epsilon>0$ there exists an open neighborhood $U$ of the identity in $\textrm{SL}_n(\mathbb R)$ such that
$$\frac{m(U\partial B_n)}{m(B_n)}<\epsilon$$
for all $n$.
\end{defn}

\begin{lemma}\label{gapwellrounded}
Let $\{KC_{T,\beta}K\}$ and $\{KC_{T,\beta}'K\}$ be as above.  Then there exists some $N\in \mathbb N$ such that the sequence $\{KC_{T,\beta}'K\}_{T>N}$ is well-rounded.
\end{lemma}

To prove this, we essentially use Lemma~\ref{wellrounded1} together with the fact that there is some $R>0$ such that $\frac{m(KC_{T,\beta}'K)}{m(KC_T'K)}>R$, and an analysis similar to that in the previous section.

\begin{proof}

First, note that  $KC_T'K$ and $KC_{T,\beta}'K$ can be viewed as convex polytopes in $\mathbb R^n$: for example, $KC_T'K$ is described in this way in (\ref{polyR2}).  In fact, the polygon corresponding to $KC_{T,\beta}'K$ is obtained from $KC_T'K$ by cutting the polygon corresponding to $KC_T'K$ by the hyperplane $j_1=j_2+\beta$: i.e., it is the intersection of the polygon corresponding to $KC_T'K$ with $j_1>j_2+\beta$.

Now, by the proof of Lemma~\ref{middletwo}, we have that for any $\eta>0$, there is some $N\in\mathbb N$
\begin{equation}\label{density1}
\frac{m(KC_{T,\beta}'K)}{m(KC_T'K)}>1-\eta=R.
\end{equation}
Let $\epsilon>0$, and let $U$ be the neighborhood of identity in $\textrm{SL}_n(\mathbb R)$ such that 
$$\frac{m(U\partial(KC_T'K))}{m(KC_T'K)}<R\epsilon/2$$
for all $T$.  From above, we have that the boundary of $KC_{T,\beta}'K$ is the union of part of the boundary of $KC_T'K$ and a polygon $P_T'$ sitting inside the hyperplane $j_1=j_2+\beta$.  By the same argument as in the proof of Lemma~\ref{middletwo}, there exists $N'\in\mathbb N$ and a neighborhood $U'$ of identity in $\textrm{SL}_n(\mathbb R)$ so that $m(U'P_T')<R\epsilon/2$ for all $T>N'$ (here it is key that $P_T'$ comes nowhere near the vertex of the polygon corresponding to $KC_T'K$ where the functions in the integrals (\ref{weirdfirstinteven}) and (\ref{weirdfirstintodd}) obtain their maxima).  Let $U''=U\cap U'$.  Then 
$$\frac{m((U''\partial(KC_{T,\beta}'K))}{m(KC_{T,\beta}'K)}<\frac{m(U''\partial(KC_T'K))+R\epsilon/2}{Rm(KC_T'K)}<\epsilon$$
for all $T>\max(N,N')$, as desired.

\end{proof}

\section{Lyapunov exponent estimates}
\label{lyapsec}
Suppose we have a (finite, for simplicity) collection $A_1, \dots, A_k$ of matrices in $\textrm{GL}_n(\mathbb Z),$ and let $P= (p_1, \dotsc, p_k)$ be a probability vector. Let $\mu = P^{\mathbb{Z}^+}$ be the associated Bernoulli measure on the set of sequences $\sigma = (1, \dotsc, k)^{\mathbb{Z}^+}.$ The \emph{Lyapunov exponent} $\lambda$ is given by 
\begin{equation}
\lambda = \lim_{n\rightarrow \infty}\frac1n \int \|A_1 \dots A_{i_n}\| d \mu(\underline{i}),
\end{equation}
where $\underline{i} = (i_n)_{n=1}^\infty.$

The celebrated result of Furstenberg and Kesten \cite{furstenberg1960products} states that $\mu$-almost everywhere, 
\begin{equation}
\lim_{n\rightarrow \infty}\frac1n \|A_1\dots a_{i_n}\| = \lambda.
\end{equation}
This result is an immediate consequence of Kingman's Subadditive Ergodic Theorem (which was discovered a few years after Furstenberg and Kesten's paper appeared). The hard part is estimating the exponent $\lambda$ (even proving that it is positive is highly nontrivial), but, luckily, when we have $(r,\epsilon)$-proximal transformations, that is relatively easy, as we show below. One of the many interpretations of the Lyapunov exponent is as the rate of divergence (or ``drift'', as it is known in the jargon) of the matrix product from the identity - the logarithm of the matrix norm is not so different from the distance in $\textrm{SL}_n.$

To estimate the drift in our setting, we need a couple of lemmas (which are essentially trivial exercises in linear algebra - we are not claiming that these are particularly original). 

First, we recall (as so many times before in this paper) the singular value decomposition: every matrix $A\in M^{n\times n}$ can be written as $U^t(A) D(A) V(A),$ where $U, V \in SO(n),$ and $D$ diagonal with entries $\lambda_1 \geq \lambda_2 \geq \dots \geq\lambda_n \geq 0.$ The $\lambda$'s are the \emph{singular values} of $A.$ We will denote the first rows of $V$ and $U$ by $v_1(A)$ and $u_1(A),$ respectively.
\begin{lemma}
\label{singleop}
Let $A \in M^{n\times n},$ with $\lambda_1(A) \gg \lambda_2(A),$ and let $w$ be a vector. Then, $$A^k w \sim \langle w, v_1(A)\rangle \langle v_1(A), u_1(A)\rangle^{k-1} \lambda_1^k u_1(A).$$
\end{lemma}
\begin{proof}
Simple computation.
\end{proof}
\begin{lemma}
\label{twoops}
Let $A, B \in \GLn(\mathbb Z)$ be such that $\lambda_1(A), \lambda_1(B), \lambda_1(A^{-1}), \lambda_1(B^{-1}) \geq \lambda \gg 1.$ Assume, further, that all the inner products $\langle u_1(A), v_1(A)\rangle, \langle u_1(A), v_1(B)\rangle, \langle u_1(A),\\ v_1(B), \langle(u_1(B), v_1(B)\rangle$ are bigger than $\epsilon$ in absolute value. Let $w(A, B)$ be a reduced word of length $k$ in $A, B, A^{-1}, B^{-1}.$ Then \begin{equation}
\| w(A, B) u_1(A)\| \geq (\epsilon \lambda)^k.
\end{equation}
\end{lemma}
\begin{proof}
Simple computation using Lemma \ref{singleop}, and the observation that $u_1(A^{-1}) = v_1(A),$ etc.
\end{proof}
We will finally need a fact:
\begin{fact}
\label{freelen}
Let $w(x, y)$ be a random product of $x, y, x^{-1}, y^{-1},$ of length $n.$ Then, with probability tending to one as $n$ goes to infinity, the reduced length of $w$ is at least $3\ell(w)/4.$
\end{fact}
Finally, we can combine the above simple observations to state:
\begin{thm}
Let $A, B$ be picked uniformly at random from the symmetrized ball of radius $R.$ Then, the Lyapunov exponent of the group generated by $A, B$ goes to infinity with $R,$ with high probability.
\end{thm}
\begin{proof} By Lemma \ref{middletwo}, we know that the the gap between the top and second singular values of $A, B$ (when acting on the appropriate exterior power) goes to infinity with $R.$ On the other hand, by equidistribution in sectors, the vectors $u_1(A), v_1(A), u_1(B), v_1(b)$ are uniformly distributed in the orthogonal group, so the absolute values of their inner products are bounded below by $\epsilon$ with probability $p(\epsilon)$ asymptotically independent of $R.$ The result now follows from Fact \ref{freelen} and Lemma \ref{twoops}.
\end{proof}
\section{Directions for further study}
\label{future}
Here, we enumerate some questions raised by this work.

First, it seems clear that the results should hold for lattices in other semisimple groups (the symplectic group comes to mind).

Although our methods do not immediately apply to arbitrary matrix norms, as demonstrated in this paper, the results should likely hold for arbitrary matrix norms, and, perhaps more interestingly, for non-archimedean height as in Aoun's work.

Finally, instead of looking at the intersection of a \emph{lattice} with a norm ball, we could look at the intersection of an arbitrary Zariski-dense subgroup. Notice that in the ``combinatorial height" setting, the results are the same as for lattices, so this should be true here too. It is clear that this would require quite different methods.

\bibliography{thin}
\bibliographystyle{plain}
\end{document}